\documentclass[a4paper,11pt,reqno]{amsart}
\usepackage{amssymb}
\usepackage{epsfig}
\usepackage{amsfonts}
\usepackage{amsmath}
\usepackage{euscript}
\usepackage{amscd}
\usepackage{amsthm}
\DeclareMathAlphabet{\mathpzc}{OT1}{pzc}{m}{it}
\usepackage{color}
\usepackage{mathtools}
\usepackage[pagebackref, hypertexnames=false, colorlinks, citecolor=red, linkcolor=blue, urlcolor=red]{hyperref}

\usepackage{marginnote}

\usepackage[normalem]{ulem}

\newcommand{\marginextend}[1]{ \addtolength{\oddsidemargin}{-#1}  \addtolength{\evensidemargin}{-#1}
  \addtolength{\textwidth}{#1}\addtolength{\textwidth}{#1}}
\newcommand{\updownextend}[1]{ \addtolength{\topmargin}{-#1}  \addtolength{\textheight}{#1}
\addtolength{\textheight}{#1}}
\marginextend{1cm}
\updownextend{0cm}
\allowdisplaybreaks[4]

\usepackage{mathrsfs}
\DeclareFontFamily{OT1}{pzc}{}
\DeclareFontShape{OT1}{pzc}{m}{it}{<-> s * [1.10] pzcmi7t}{}
\DeclareMathAlphabet{\mathpzc}{OT1}{pzc}{m}{it}

\DeclareSymbolFont{SY}{U}{psy}{m}{n}
\DeclareMathSymbol{\emptyset}{\mathord}{SY}{'306}

\theoremstyle{plain}

\newtheorem{thm}{Theorem}[section]
\newtheorem*{thm*}{Theorem}
\newtheorem{cor}[thm]{Corollary}
\newtheorem{lem}[thm]{Lemma}
\newtheorem{prop}[thm]{Proposition}
\newtheorem{defn}[thm]{Definition}
\newtheorem{rem}[thm]{Remark}
\newtheorem{ex}[thm]{Example}

\newtheoremstyle{named}{}{}{\itshape}{}{\bfseries}{.}{.5em}{#1 \thmnote{#3}}
\theoremstyle{named}

\numberwithin{equation}{section}

\def\C{{\mathbb C}}

\def\norm#1{\left\|{#1}\right\|}

\def\v{\varphi}

\def\ov{\overline}

\def\m{\mathcal}
\def\mb{\mathbb}
\def\mr{\mathrm}
\def\a{\alpha}

\def\wi{\widetilde}

\def\w{\widehat }

\def\beq{\begin{eqnarray}}
\def\eeq{\end{eqnarray}}
\def\beqa{\begin{eqnarray*}}
\def\eeqa{\end{eqnarray*}}

\def\ov{\overline}
\def\bl{\boldsymbol}

\def\i{\prime}

\def\bl{\boldsymbol}

\newcommand{\be}{\begin{equation}}
\newcommand{\ee}{\end{equation}}
\newcommand{\bea}{\begin{eqnarray}}
\newcommand{\eea}{\end{eqnarray}}
\newcommand{\Bea}{\begin{eqnarray*}}
\newcommand{\Eea}{\end{eqnarray*}}
\newcommand{\inner}[2]{\langle #1,#2 \rangle }%

\newcounter{cnt1}
\newcounter{cnt2}
\newcounter{cnt3}
\newcommand{\blr}{\begin{list}{$($\roman{cnt1}$)$}
 {\usecounter{cnt1} \setlength{\topsep}{0pt}
 \setlength{\itemsep}{0pt}}}
\newcommand{\bla}{\begin{list}{$($\alph{cnt2}$)$}
 {\usecounter{cnt2} \setlength{\topsep}{0pt}
 \setlength{\itemsep}{0pt}}}
\newcommand{\bln}{\begin{list}{$($\arabic{cnt3}$)$}
 {\usecounter{cnt3} \setlength{\topsep}{0pt}
 \setlength{\itemsep}{0pt}}}
\newcommand{\el}{\end{list}}

\let\oldbibliography\thebibliography
\renewcommand{\thebibliography}[1]{\oldbibliography{#1}
\setlength{\itemsep}{2pt}} 

\newcommand{\dnorm}[1]{\left\lVert#1\right\rVert_{\m D}}
\newcommand{\mnorm}[1]{\left\lVert#1\right\rVert_{m}}

\begin{document}
\title[Multiplication operator on Bergman space] {Multiplication operator on the Bergman space by a proper holomorphic  map}

\author[Ghosh]{Gargi Ghosh}
\address[Ghosh]{Department of Mathematics and Statistics, Indian Institute of Science Education and Research Kolkata, Mohanpur 741246, Nadia, West Bengal, India}

\email[Ghosh]{gg13ip034@iiserkol.ac.in}
\thanks{The work of G. Ghosh is supported by Senior Research Fellowship funded by CSIR}


\maketitle
\begin{abstract}
Suppose that $\bl f := (f_1,\ldots,f_d):\Omega_1\to\Omega_2$ is a proper holomorphic map between two bounded domains in $\mathbb C^d.$ We show that the multiplication operator (tuple) $\mathbf M_{\bl f}=(M_{f_1},\ldots, M_{f_d})$ on the Bergman space $\mathbb A^2(\Omega_1)$ admits a non-trivial minimal joint reducing subspace,  say $\mathcal M.$ and the restriction of $\mathbf M_{\bl f}$ to $\mathcal M$ is unitarily equivalent to Bergman operator on $\mb A^2(\Omega_2).$
\end{abstract}

\section{Introduction}
For $d\geq 1,$ let $\Omega$ be a bounded domain  in $\mb C^d.$ The Bergman space on $\Omega,$ denoted by $\mb A^2(\Omega)$ is the subspace of holomorphic functions in $L^2(\Omega)$ with respect to Lebesgue measure on $\Omega.$ The Bergman space  $\mb A^2(\Omega)$ is a  Hilbert space with a reproducing kernel, called the Bergman kernel.  We study multiplication operators on $\mb A^2(\Omega)$ induced by holomorphic functions. Thus for a bounded holomorphic functions $\bl g=(g_1,\ldots, g_d):\Omega\to\mb C^d,$  we define 
$M_{g_i}:\mb A^2(\Omega)\to\mb A^2(\Omega)$
by
\Bea
M_{g_i}\v=g_i\v, \,\,  \v\in\mb A^2(\Omega) \text{~for~} i=1,\ldots, d.
\Eea
Clearly, each $M_{g_i}$ is a bounded linear operator.
Let $\mathbf M_{\bl g}$ denote the tuple of operators $(M_{g_1}.\ldots,M_{g_d}).$ We put $\mathbf M:=(M_1,\ldots, M_d),$ where $M_i$ denotes the multiplication by the $i$-th coordinate function of $\Omega$ on $\mb A^2(\Omega)$ for $i=1,\ldots, d,$ $\mathbf M$ is called the {\it Bergman operator} on $\mb A^2(\Omega)$ or the Bergman operator {\it associated to} a domain $\Omega.$ A {\it joint reducing subspace} for a  tuple 
$\mathbf T=(T_1,\ldots,T_d)$ of operators on a Hilbert space $\m H$ is a closed subspace $\m M$ of $\m H$ such that 
\Bea
T_i\m M\subseteq\m M \text{~ and~} T_i^*\m M\subseteq\m M \text{~for~} i=1,\ldots,d.
\Eea 
It is easy to see that $\mathbf T$ has a  joint reducing subspace of  if and only if there exists an orthogonal projection $P$ ($P:\m H\to\m H, P^2=P=P^*$) such that 
\Bea
PT_i=T_iP \text{~for~} i=1,\ldots, d.
\Eea 
Here $\m M=P\m H$ is the corresponding joint reducing subspace of $\mathbf T.$
A joint reducing subspace $\m M$ of $\mathbf T$  is called {\it minimal} or {\it irreducible} if only joint reducing subspaces contained in $\m M$ are $\m M$ and $\{0\}.$ In particular, the operator tuple $\mathbf T$ is called ${\it irreducible}$ if only  joint reducing subspaces contained in $\m H$ are $\m H$ and $\{0\}.$

It is natural to start with a bounded biholomorphic multiplier $\bl f$ on $\mb A^2(\Omega).$  It is easy to see that $\mathbf M_f$ is irreducible as follows. Let $\Omega_1,\Omega_2$ be bounded domains in $\mb C^d$ and $\bl f=(f_1,\ldots,f_d):\Omega_1\to\Omega_2$ be a biholomorphism. Let $J_{\bl f}=\det\big(\!\!\big(\frac{\partial f_j}{\partial z_i}\big)\!\!\big)_{i,j=1}^d$ denote the jacobian of $\bl f.$  The linear map $U_{\bl f}:\mb A^2(\Omega_2)\to\mb A^2(\Omega_1)$ defined by 
\Bea
U_{\bl f}\v=(\v\circ \bl f)J_{\bl f} \text{~for~} \v\in \mb A^2(\Omega_2)
\Eea
is a unitary satisfying $U_fM_i=M_{f_i}U_f$ for $i=1,\ldots,d.$   Therefore, $\mathbf M_{\bl f}$ is unitarily equivalent to the Bergman operator on $\mb A^2(\Omega_2).$ Since the Bergman operator associated to any domain is irreducible, it follows that $\mathbf M_{\bl f}$ is irreducible.

This is the motivation for considering the well behaved class of bounded multipliers induced by proper holomorphic maps.
 A holomorphic map $\bl f:\Omega_1\to\Omega_2$ is said to be {\it proper} if $\bl f^{-1}(K)$  is compact in $\Omega_1$ whenever $K\subseteq\Omega_2$ is compact. Clearly, a biholomorphic map proper, basic properties of proper holomorphic maps are discussed in \cite[Chapter 15]{MR2446682}. The problem of understanding the reducing subspaces of multiplication operators by finite Blaschke products (same as the class of proper holomorphic self maps of the unit disc $\mb D$ in $\mb C$)  on the Bergman space  $\mb A^2(\mb D)$ in  and parametrizing the number of minimal reducing subspaces has been studied profusely, see \cite{MR3363367} and references therein. Analogous problems for several variables are pursued in \cite{HZ,MR3902057}. In this paper, we give a bare hand derivation of the two results:
\begin{enumerate}
    \item [1.] If $\Omega_1,\Omega_2\subseteq\mb C^d$ are two bounded domains  and $f=(f_1,\ldots, f_d):\Omega_1\to\Omega_2$ is a proper holomorphic mapping, then $\mathbf M_f=(M_{f_1},\ldots,M_{f_d})$ has a non-trivial minimal joint reducing subspace $\m M.$
    \item[2.] The restriction of $\mathbf M_f$ to $\m M$ is unitarily equivalent to the Bergman operator on $\mb A^2(\Omega_2).$
\end{enumerate}
We draw attention to the simplicity of our methods and to the fact that we do not require any regularity on the domains $\Omega_1,\Omega_2$ to obtain the aforementioned results. Therefore, it allows us to specialize the domains $\Omega_1,\Omega_2$ to be different standard domains and obtain explicit description  of the non-trivial minimal reducing subspace and the restriction operator ${\mathbf M_f}_{\vert_{\m M}}.$ In particular, we specialize to
\begin{enumerate}
\item [1.] $\Omega_2$ to be a complete Reinhardt domain and $\Omega_1$ is any bounded domain.
\vspace{.1 cm}
\item[2.] $\Omega_1=\Omega_2=\mb D,$ the unit disc in $\mb C,$ the result obtained here is one of the main results in  \cite[Theorem 25]{MR2503238}, \cite[Theorem 15, p. 393]{MR1969797} and the main result in \cite{MR2068435}.

\item[3.] $\Omega_1=\Omega_2=\mb D^d,$ the unit polydisc in $\mb C^d.$  
\vspace{.1 cm}
\item[4.] $\Omega_1=\Omega_2=\mb G_d,$ the symmetrized polydisc in $\mb C^d.$ This domain has been studied extensively during last two decades from the viewpoint of function theory and operator theory, for example, see \cite{MR3188714, MR2142182, MR2135687} and references therein. 
\end{enumerate}

For two domains $D_1$ and $D_2$ in $\mb C^d$ and an analytic covering $p : D_1\to D_2,$ a {\it deck transformation} or an {\it automorphism} of $p$  is a biholomorphism $h : D_1 \to D_1$ such that $p \circ h = p.$ The set of all deck transformations of $p$ forms a group. We call it the group of deck transformations of $p$ or group of automorphisms of $p$ and denote it by ${\rm Deck}(p)$ or ${\rm Aut}(p).$ It is well known that a proper holomorphic mapping $f:\Omega_1\to\Omega_2$ is a (unbranched) analytic covering 
\Bea
f:\Omega_1\setminus f^{-1}f( V_f)\to\Omega_2\setminus f(V_f),
\Eea
where $V_f:=\{z\in\Omega_1:J_f(z)=0\}$ is the {\it branch locus} of $f$ \cite[Subsection 15.1.10]{MR2446682}.
The group ${\rm Deck}(f)$ of this (unbranched) analytic covering is called the group of deck transformations of the proper map $f:\Omega_1\to\Omega_2.$
An analytic covering $p$ is called {\it Galois} if Deck$(p)$ acts transitively on the fibre set  $p^{-1}(\{ w\})$ of $ w\in D_2$ for some (and thus for all) $ w \in D_2.$ Equivalently, cardinality of the set $p^{-1}(\{ w\})$ equals to the order of Deck$(p).$ In other words, for every $w\in D_2,$ the set $p^{-1}(\{w\})$ is a ${\rm Deck} (p)$-orbit, and vice versa. Briefly, $D_2=D_1/ {\rm Deck} (p).$ Hence an analytic cover $p$ is Galois if and only if $D_1/{\rm Deck}(p) \cong D_2.$

We recall a notion from \cite{MR807258}. We say that  $f:D_1\to D_2$ is {\it factored by automorphisms} if there exists a finite subgroup $G\subseteq {\rm  Aut}(D_1)$ such that
\bea\label{factor}
f^{-1}f(z)=\bigcup_{\rho\in G} \rho(z) \text{~for all~} z\in D_1.
\eea
Form this it follows that $f$ is $G$-invariant, that is, $f\circ \rho=f$ for $\rho\in G.$  If such a group $G$ exists  then $f$ factors a $\tilde f\circ\eta,$ where $\eta:D_1\to D_1/G$ is a quotient map and $\tilde f:D_1/G\to D_2$ is a biholomorphism. That is, $G\cong {\rm Deck} (f).$

Now we describe a procedure developed in \cite{MR3906291, BDGS} to obtain joint reducing subspaces of $\mathbf M_f$ acting on a Hilbert space $\m H\subseteq {\rm Hol}(\Omega_1,\mb C^d)$ (in particular,   $\mb A^2(\Omega_1)$) with $G$-invariant reproducing kernel $K,$ that is.
\Bea
K(\rho\cdot z ,\rho\cdot  w)=K(z,  w) \text{~for~} z, w\in\Omega_1, \rho\in G. 
\Eea
for a proper holomorphic map $f=(f_1,\ldots, f_d):\Omega_1\to\Omega_2$ which factors through automorphisms $G \subseteq {\rm Aut}(\Omega_1),$ here $G \subseteq {\rm Aut}(\Omega_1)$ is the group which appears in \eqref{factor}. If  $\w G$ denotes the set of equivalence class of irreducible representations $G,$ then it is shown that 
\Bea
\m H=\bigoplus_{\varrho\in\w G}\mb P_\varrho\m H,
\Eea
where $\mb P_\varrho$ is an orthogonal projection satisfying $\mb P_\varrho M_{f_i}=M_{f_i}\mb P_\varrho$ for $\varrho \in\w G$ and $i=1,\ldots, d.$ Thus we obtain a family $\{\mb P_\varrho\m H:\varrho\in\w G\}$ of joint reducing subspaces for the multiplication tuple $\mathbf M_f.$ 

A factorization of a proper holomorphic map does not always exist (see \cite[p. 223]{MR1131852}, \cite[Section 4.7, p. 711]{MR667790}). We have justified by explicit examples that proper holomorphic self-maps of $\mb D,\mb D^d$ and $\mb G_d$ do not factor through automorphisms, in general, by showing their groups of deck transformations are trivial ( since $G\cong {\rm Deck}(f),$ if $f:\Omega_1\to\Omega_2$ factors through automorphisms $G\subseteq {\rm Aut}(\Omega_1).$) We are able to produce a non-trivial joint minimal reducing subspace $\m M$ of $\mathbf M_f$ acting on $\mb A^2(\Omega_1)$ and describe the restriction operator ${\mathbf M_f}_{\vert_{\m M}},$  even if $f$ does not factor through automorphisms.

\section{Proper holomprphic maps and reducing subspace}
Let $f:\Omega_1\to\Omega_2$ be a proper holomorphic map of multiplicity $m$ between two bounded domains in $\mb C^d.$ We denote the Bergman space on $\Omega_i$ by $\mb A^2(\Omega_i)$ 
and the reproducing kernels by $K_i$ for $i=1,2,$ respectively. 

Let  $\Gamma_f:\mb A^2(\Omega_2)\to\mb A^2(\Omega_1)$ be the linear map  defined by the rule:
\bea
\label{genp}\Gamma_f  \psi=\frac{1}{ \sqrt m}  (\psi\circ f)J_f, \,\, \psi\in \mb A^2(\Omega_2),
\eea
where $J_f$ is the jacobian of the proper map $f.$
Next we note that $f(N)$ has measure zero with respect to the Lebesgue measure on $\mb C^n,$ where $N=Z(J_f)$ is the zero set of $J_f.$ By the change of variables formula, we obtain,
\bea\label{isome}
\int_{\Omega_1}\vert J_f\vert^2\vert \v\circ f\vert^2 dV \nonumber &=&\int_{\Omega_1-f^{-1}f(N)}\vert J_f\vert^2\vert \v\circ f\vert^2 dV \\
&=& m \int_{\Omega_2-f(N)}\vert \v\vert^2 dV=m \int_{\Omega_2}\vert \v\vert^2 dV  ,
\eea
where $\v\in \mb A^2(\Omega_2)$ and $dV$ is the Lebesgue measure.
This shows that the map $\Gamma_f$ is an isometry. 

In the following proposition, we provide a criterion for a proper holomorphic map $f$ to be biholomorphic in terms of the isometry $\Gamma_f.$ 

\begin{prop}
The linear map $\Gamma_f:\mb A^2(\Omega_2)\to \mb A^2(\Omega_1)$ is surjective if and only if  $f$ is biholomorphic.
\end{prop}
\begin{proof}

If $f$ is biholomorphic, then $f^{-1}:\Omega_2\to\Omega_1$ is also biholomorphic. It follows from the chain rule that $\Gamma_{f^{-1}}:\mb A^2(\Omega_1)\to\mb A^2(\Omega_2)$ is the inverse of  $\Gamma_f.$ Hence $\Gamma_f$ is surjective.

Conversely, suppose $\Gamma_f$ is surjective. 
Arguing by contradiction, suppose that $f$ is not bijective. Since $f$ is a proper holomorphic map, it is always surjective, so $f$ fails to be injective by our assumption. Suppose $\bl w$ and $\bl w'$ are two distinct points in $\Omega_1$ such that $f(\bl w) = f(\bl w') = \bl u.$ Since $\bl w,\bl w'$ are distinct, there exists at least one $1 \leq i \leq n$ such that $w_i \neq w'_i.$ Without loss of generality, let consider $w_1 \neq w'_1.$  We consider two holomorphic functions $g_i : \Omega_1 \to \mb C$ for $i=1,2,$ where $g_1(\bl z) = 1 {\rm ~and ~} g_2(\bl z) = z_1$ for all $\bl z=(z_1,\ldots,z_n) \in \Omega_1.$ Since $\Omega_1$ is a bounded domain, both $g_i$'s are in $\mb A^2(\Omega_1)$. Also, $\Gamma_f$ is surjective, so there exist $\phi_i \in \mb A^2(\Omega_2),$ for $i=1,2$ such that $\Gamma_f(\phi_i) = \frac{1}{ \sqrt m} J_f ~ \phi_i \circ f = g_i,$ $i=1,2.$ Note that $\frac{1}{ \sqrt m} J_f(\bl w) \phi_1(\bl u) = 1 = \frac{1}{ \sqrt m} J_f(\bl w') \phi_1(\bl u),$ alternatively, $\phi_1(\bl u)\big(J_f(\bl w)-J_f(\bl w')\big) = 0.$ Since the product of two complex numbers $\frac{1}{ \sqrt m} J_f(\bl w)$ and $\phi_1(\bl u)$ is equal to $1,$ so $\phi_1(\bl u)$ cannot be equal to $0.$ Then $J_f(\bl w)= J_f(\bl w').$ This indicates that $\frac{1}{ \sqrt m} \phi_2(\bl u)\big(J_f(\bl w)-J_f(\bl w')\big)= 0 = w_1 - w'_1 \neq 0,$ a contradiction. Hence, our assumption was wrong and $f$ is injective. Since $f$ is a bijective holomorphic map, it is a biholomorphism, by \cite[p. 37, Theorem 1.6.5.]{L}.

\end{proof}

Now onward by a proper map we mean a proper holomorphic map which is not
biholomorphic. For a function $f=(f_1.\ldots, f_d):\Omega_1\to\Omega_2,$ the symbol $\mathbf M_f$ is defined to be the commuting tuple $(M_{f_1},\ldots,M_{f_d}).$  Let $\mathbf M:=(M_{1},\ldots,M_{d})$ denote the commuting tuple of multiplication operators by the coordinate functions, that is,  $M_{i}$ denotes multiplication by the $i$-th coordinate function, $ i=1,\ldots, d.$ 

 
The following theorem exhibits explicitly a non-trivial joint  reducing subspace (consequently, at least two) of the tuple of multiplication operators $\mathbf M_f$ by a proper map $f:\Omega_1\to\Omega_2$  on the Bergman space $\mb A^2(\Omega_1).$ This result strengthens \cite[Theorem 1.4]{HZ} by displaying the joint reducing subspace of $\mathbf M_f$ explicitly without making any reference to operator algebras.
\begin{thm}\label{redsub}
 Suppose that $\Omega_1, \Omega_2$ are bounded domains in $\mb C^d$ and $f: \Omega_1 \to\Omega_2$ is a proper holomorphic map of multiplicity $m.$ If $\Gamma_{f} : \mb A^2(\Omega_2) \to \mb A^2(\Omega_1)$ is defined by 
\Bea
\Gamma_{f}  \psi=\frac{1}{ \sqrt m} (\psi\circ f)  J_{f}, \,\, \psi\in \mb A^2(\Omega_2),
\Eea
where $J_{f}$ is the jacobian of the proper map $f.$ Then $\Gamma_{f}\big(\mb A^2(\Omega_2)\big)$ is a joint reducing subspace for the commuting tuple $\mathbf M_f = (M_{f_1},\ldots,M_{f_d})$ on $\mb A^2(\Omega_1).$
\end{thm}

\begin{proof}
Being the image of a Hilbert space under the isometry $\Gamma_f,$  $\Gamma_f \big(\mb A^2(\Omega_2)\big)$ is  a closed subspace of $\mb A^2(\Omega_1).$ The orthogonal projection $P$ from $\mb A^2(\Omega_1)$ onto $\Gamma_{f}\big(\mb A^2(\Omega_2)\big)$ is given by the following formula \cite[p. 551]{MR3133729}
\bea \label{projec}
P\v = \frac{1}{m} \displaystyle\sum_{k=1}^m (\v \circ f^k \circ f) J_{f^k \circ f}, \,\, \v \in \mb A^2(\Omega_1),
\eea
where $\{f^j\}_{j=1}^m$ are the local inverses of $f.$ Let $e_i : \Omega_2 \to \mb C$  denote the coordinate projections  defined by $e_i(\bl w) = w_i$ for $i=1,\ldots, d.$ So that 
\bea\label{projeq}
e_i\big(f(\bl z)\big) = (e_i \circ f)(\bl z) = f_i(\bl z).
\eea 
\bea \label{projeq1}
(PM_{f_i})\v =P(f_i \v)=\frac{1}{m}\sum_{k=1}^m(f_i \circ f^k \circ f) ~(\v \circ f^k \circ f)~ J_{f^k \circ f}.
\eea
Repetitive use of the relation \eqref{projeq} and the fact $(f \circ f^k )(\bl z) = \bl z,$ lead us to 
$$f_i \circ f^k \circ f = (e_i \circ f) \circ (f^k \circ f) =e_i \circ (f \circ f^k) \circ f= e_i \circ f =f_i.$$
Hence  from \eqref{projeq1}, it  follows immediately that
$$(PM_{f_i})\v = f_i P\v =( M_{f_i}P)\v \text{~for~} \v \in \mb A^2 (\Omega_1).$$
This shows that $P$ commutes with each component of $(M_{f_1},\ldots,M_{f_d}).$ Hence $\Gamma_{f}\big(\mb A^2(\Omega_2)\big)$ is a joint reducing subspace of the tuple $\mathbf M_{f}.$
\end{proof}
\begin{rem}\rm
Suppose $\alpha \in \mathcal C(\Omega_2, \mb R_{>0}).$ In \cite{MR3133729}, the author describes the map $\Gamma_f : \mb A_\alpha^2(\Omega_2) \to \mb A_{\alpha \circ f}^2(\Omega_1)$ for weighted Bergman spaces. In that case, the projection $P$ in \eqref{projec} corresponds to the subspace $\Gamma_{f}\big(\mb A_\alpha^2(\Omega_2)\big)$ of $\mb A_{\alpha \circ f}^2(\Omega_1).$ In the similar way as above, $P$ intertwines with $\mathbf M_f$ on $\mb A_{\alpha \circ f}^2(\Omega_1)$. Hence, the above result can be stated in much more generality, that is, $\Gamma_{f}\big(\mb A_\alpha^2(\Omega_2)\big)$ is a joint reducing subspace for the commuting tuple $\mathbf M_f$ on $\mb A_{\alpha \circ f}^2(\Omega_1).$
\end{rem}

In the next theorem, we offer a canonical description of the restriction of $\mathbf M_f$ to the joint reducing subspace  ${\rm ran} ~\Gamma_f.$ 



\begin{thm}\label{bergshift}
The restriction of $\mathbf M_{f} $ to  $\Gamma_f\big(\mb A^2(\Omega_2)\big)$ is unitarily equivalent to $\mathbf M $ on ${\mb A^2(\Omega_2)}.$ Consequently, the restriction of $\mathbf M_{f} $ to  $\Gamma_f\big(\mb A^2(\Omega_2)\big)$ is irreducible.
\end{thm}
\begin{proof}
The map $\Gamma_{f} : \mb A^2(\Omega_2) \to \Gamma_f\big(\mb A^2(\Omega_2)\big)$ is  surjective and it intertwines $M_{f_i} $ on $  \Gamma_{f}\big(\mb A^2(\Omega_2)\big)$ and $M_{i}$ on $\mb A^2(\Omega_2)$ for all $i=1,\ldots,d,$ because
$$M_{f_i}\Gamma_f \psi = f_i( \psi \circ f )J_f = (e_i\circ f)( \psi \circ f) J_f = \Gamma_f(e_i \psi) = \Gamma_{f} M_{i} \psi  \text{~for~}\psi \in \mb A^2(\Omega_2).$$
 By Equation \eqref{isome}, $\Gamma_f$ is an isometry. This completes the proof.
\end{proof}

\begin{rem}
For the sake of completeness, we state the version of above theorem for weighted Bergman spaces. It states that the restriction of $\mathbf M_{f} $ to  $\Gamma_f\big(\mb A_\alpha^2(\Omega_2)\big)$ is unitarily equivalent to $\mathbf M $ on ${\mb A_\alpha^2(\Omega_2)}$, and hence,  irreducible.
\end{rem}

The following corollary follows immediately from the Theorem above.
\begin{cor}\label{self}
If $\Omega\subseteq\mb C^d$ is a domain such that $f$ is a proper holomorphic self-map of $\Omega$ of multiplicity $m.$ Then the restriction of $\mathbf M_f$ to the joint reducing subspace  $\Gamma_f\big(\mb A^2(\Omega)\big)$ is unitarily equivalent to $\mathbf M$ on $\mb A^2(\Omega).$  
\end{cor}
In the next section, we provide examples where the Corollary above is applicable. 
 
Before proceeding further, we recall a relevant definition.
\begin{defn}
A domain $D$ in $\mb C^d$ is called a complete Reinhardt domain if whenever $\bl z\in D,$ the closed polydisc $\{\vert w_j\vert\leq\vert z_j\vert:j=1,\ldots,d\}$ is also contained in $D.$
\end{defn}
In such a domain $D,$ the monomials form a complete orthogonal system for $\mb A^2(D).$ Let $\mb Z_+$ denote the set of nonnegative integers and $\bl \a=(\a_1,\ldots,\a_d)\in\mb Z_+^d$ be a multi-index. For $\bl z=(z_1,\ldots,z_d)\in\mb C^d, \bl z^{\bl\a}:=\prod_{j=1}^dz_j^{\a_j}.$
We make a note of an observation in the following proposition.
\begin{prop}\label{onb}
 Suppose that $\Omega_1, \Omega_2$ are bounded domains in $\mb C^d$ such that $\Omega_2$ is a complete Reinhardt domain and $f=(f_1,\ldots,f_d): \Omega_1 \to\Omega_2$ is a proper holomorphic map of multiplicity $m.$ Then $\Big\{\frac{1}{\sqrt m\Vert \bl z^{\bl \a}\Vert}J_ff^{\bl \a}\Big\}_{\bl\a\in\mb Z_+^d}$ is an orthonormal basis for $\Gamma_f\big(\mb A^2(\Omega_2)\big).$
\end{prop}
\begin{proof}
We observe that $\Gamma_f$ is an isometry onto its range, hence it maps an  orthonormal basis of $\mb A^2(\Omega_2)$ to an orthonormal basis in its range. Since $\Omega_2$ is a complete Reinhardt domain, $\Big\{\frac{\bl z^{\bl \a}}{\Vert\bl z^{\bl \a}\Vert}\Big\}_{\bl\a\in\mb Z_+^d}$ is an orthonormal basis of $\mb A^2(\Omega_2).$ The result follows noting that the desired orthonormal basis is the image of $\Big\{\frac{\bl z^{\bl \a}}{\Vert\bl z^{\bl \a}\Vert}\Big\}_{\bl\a\in\mb Z_+^d}$ under $\Gamma_f.$
\end{proof}

The following corollary is immediate from  Corollary \ref{self} and Proposition \ref{onb}.
\begin{cor}\label{selfR}
If $\Omega\subseteq\mb C^d$ is a complete Reinhardt domain such that $f$ is a proper holomorphic self-map of $\Omega$ of multiplicity $m.$ Then the restriction of $\mathbf M_f$ to the joint reducing subspace  $\Gamma_f\big(\mb A^2(\Omega)\big)$ is unitarily equivalent to $\mathbf M$ on $\mb A^2(\Omega).$ Moreover, 
\Bea
 \Gamma_{f}\big(\mb A^2(\Omega)\big)=\ov{\rm span} \{\mr J_f f^\a:\a\in\mb Z_+^d\} \text{~and~} \Big\{\frac{1}{\sqrt m\Vert\bl z^{\bl\a}\Vert}J_ff^\a\Big\}_{\a\in\mb Z_+^d}
 \Eea
 is an orthonormal basis of $\Gamma_f\big(\mb A^2(\Omega)\big).$
 \end{cor}
 

\section{Applications}

\subsection{Example (Unit disc)}
In order to describe all possible  proper holomorphic self-maps of the unit disc $\mb D,$ we recall a definition.
\begin{defn}
A {\it finite Blaschke product} is a rational function of the form
\Bea
{\rm B}(z)=e^{i\theta}\displaystyle\prod_{j=1}^n\bigg(\frac{z-a_j}{1-\bar a_j z}\bigg)^{k_j},
\Eea
where $a_1,\ldots, a_n \in\mb D$ are the distinct zeros of $\rm B$ with multiplicities $k_1,\ldots,k_n,$ respectively and $\theta\in\mb R.$
\end{defn}

Finite Blaschke products are examples of proper holomorphic self-maps of the unit disc $\mb D.$ In fact, the following proposition shows that they constitute the set of all possible proper holomorphic self-maps of the unit disc $\mb D,$ whose proof can be found in  \cite[Remarks 3, p. 142]{MR1803086}.

\begin{prop}\label{propdisk}
Let $f:\mb D\to\mb D$ be a proper holomorphic map. Then $f$ is a finite Blaschke product.
\end{prop}


In \cite[Example 2.1, p. 334]{MR1315866}, authors produce one example of finite Blaschke product given by
\bea\label{B1}
\wi{\rm B}(z) = z^4 \bigg(\frac{z- 1/2}{1-{z}/{2}}\bigg)^2.
\eea
such that ${\rm Deck}(\wi {\rm B})$ is trivial. In \cite[4.7, p. 711]{MR667790}, Rudin exhibits another example of Blaschke product
\bea\label{B2}
\w{\rm B}(z) = \prod_{j=1}^3 \frac{z-a_j}{1-\bar a_j z},  
\text{~where~} a_1 = -\frac{1}{2}, ~ a_2 = 0, ~ a_3 =\frac{3}{4},
\eea
with trivial ${\rm Deck}(\w {\rm B}).$ It is clear that any finite Blaschke product B which has either $\wi{\rm B}$ or $\w{\rm B}$ as a factor will have trivial Deck(B).

Theorem \ref{redsub} and Corollary \ref{selfR} yield the following result which is one of the main results in \cite{MR2503238}. The same result was also obtained differently in \cite[Theorem 15, p. 393]{MR1969797} and it is the main result in \cite{MR2068435}. Put $e_n(z)=\sqrt{n+1}z^n$ for $n\geq 0$ and $z\in\mb D,$ recall that $\{e_n\}_{n=0}^\infty$ forms an orthonormal basis for $\mb A^2(\mb D).$ The multiplication operator $M$ by $z$ on $\mb A^2(\mb D)$ is a weighted shift operator, known as the Bergman shift $Me_n=\sqrt{\frac{n+1}{n+2}}e_{n+1}.$
\begin{thm}\label{D}
 Suppose that $\mr B$ is a finite Blaschke product on $\mb D$ of order $m.$ Then $\Gamma_{{\rm B}}\big(\mb A^2(\mb D)\big)$ is a non-trivial minimal reducing subspace of $M_{\mr B}$ on $\mb A^2(\mb D).$ Moreover, the restriction of $M_\mr B$ to $\Gamma_{{\rm B}}\big(\mb A^2(\mb D)\big)$ is unitarily equivalent to the Bergman shift. In fact,
 \Bea
 \Gamma_{{\rm B}}\big(\mb A^2(\mb D)\big)=\ov{\rm span} \{\mr B^n\mr B^\i:n\geq 0\} \text{~and~} \Bigg\{\sqrt{\frac{n+1}{m}}\mr B^n\mr B^\i\Bigg\}_{n=0}^\infty 
 \Eea
 is an orthonormal basis of $\Gamma_{{\rm B}}\big(\mb A^2(\mb D)\big).$
 
\end{thm}

\subsection{Example (polydisc)}
Proper holomorphic self maps of polydisc is described by the following theorem (see \cite{MR733691}, \cite{MR1324108}).
\begin{thm}\label{B}
 If $\Omega_1,\ldots,\Omega_d,\Delta_1,\ldots,\Delta_d \subset \C$ are bounded domains and if $f : \Omega_1 \times \cdots \times \Omega_d \to \Delta_1 \times \cdots \times \Delta_d$ is a proper mapping, then there exist a permutation $\sigma$ of $\{1,\ldots,d\}$ and proper maps $f_i : \Omega_{\sigma(i)} \to \Delta_j$ such that $$f(z_1,\ldots,z_d) = \big(f_1(z_{\sigma(1)}),\ldots,f_n(z_{\sigma(d)})\big).$$
\end{thm}
Since Proposition \ref{propdisk} reads that any proper holomorphic self-map of the open unit disc $\mb D\subseteq \mb C$ is given by some finite Blaschke product, hence by Theorem \ref{B} we get that any proper holomorphic self-map $\mathbf B$ of the  polydisc $\mb D^d$ is given by 
\bea\label{Bl}
{\mathbf B}(\bl z) = \big({\rm B}_1 (z_1),\ldots, {\rm B}_d (z_d)\big),
\eea
where each ${\rm B}_i$ is a finite  Blaschke product. 

Before proceeding further, we produce examples of proper holomorphic self-maps $\mathbf B$ of $\mb D^d$ with trivial Deck($\mathbf B$). 
\begin{prop}
There is a proper holomorphic self-map $\mathbf B$ of $\mb D^d$ with ${\rm Deck}(\mathbf B)=\{{\rm identity}\}.$
\end{prop}
\begin{proof}
Choose the finite Blaschke product $\w{\rm B}$ given in \eqref{B2} and define $\mathbf B:\mb D^d\to\mb D^d$ by
\bea\label{MR733691}
\mathbf B(\bl z)=\big(\w{\rm B} (z_1),\ldots, \w{\rm B} (z_d)\big)
\eea
If  $\bl\v\in{\rm Deck}(\mathbf B)$ for some $\bl\v\in {\rm Aut}(\mb D^d),$ we claim that $\bl\v={\rm identity}.$  It is well known that ${\rm Aut}(\mb D^d)\cong {{\rm Aut}(\mb D)}^d\rtimes\mathfrak S_d,$ where $\mathfrak S_d$ is the permutation group on $d$ symbols. So $\bl\v=(\v_1,\ldots,\v_d),$ where $\v_i\in {\rm Aut}(\mb D)$ for $i=1,\ldots,d.$ If $\mathbf B\circ\bl\v=\mathbf B,$ then \Bea
\mathbf B\big(\v_1(z_1),\ldots,\v_d(z_d)\big)=\mathbf B(z_1,\ldots,z_d) \text{~for all~} \bl z=(z_1,\ldots,z_d)\in\mb D^d.
\Eea
By our choice of $\mathbf B$ as in \eqref{MR733691}, we get
\Bea
\Big(\w{\rm B}\big(\v_1(z_1)\big),\ldots,\w{\rm B}\big(\v_d(z_d)\big)\Big)=\big(\w{\rm B}(z_1),\ldots, \w{\rm B}(z_1)\big) \text{~for all~} \bl z=(z_1,\ldots,z_d)\in\mb D^d.
\Eea
Hence $\w{\rm B}\circ\v_i=\w{\rm B}$ for $i=1,\ldots,d.$ Therefore, by the choice of $\w{\rm B}$ in \eqref{B2}, we have $\v_i={\rm identity}$ for $i=1,\ldots, d.$ Thus $\bl\v={\rm identity}.$
\end{proof}
\begin{rem}\rm
From the proof of the Proposition above it is clear that if  $\mathbf B$ is a proper holomorphic self-map each of whose components contains $\w{\rm B}$ as a factor then Deck($\mathbf B$) is trivial.
\end{rem}

 Put $e_{\bl m}(\bl z)=\sqrt{\prod_{i=1}^d(m_i+1)}z^{\bl m}$ for $\bl z\in \mb D^d.$ We recall that $\{e_{\bl m}\}_{\bl m\in\mb Z^d}$ is an orthonormal basis for the Bergman space $\mb A^2(\mb D^d).$ If $M_i$ denotes the multiplication operator by $z_i$ on $\mb A^2(\mb D^d),$ then $\mathbf M=(M_1,\ldots,M_d)$  is a several variable weighted shift on $\mb A^2(\mb D^n),$ also called Bergman multishift whose  multi-weight sequence is $\Big(\sqrt{\frac{m_1+1}{m_1+2}},\ldots, \sqrt{\frac{m_d+1}{m_d+2}} \Big).$ Let $\mathbf M_{\mathbf B}$ be the operator tuple $(M_{\mr B_1},\ldots M_{\mr B_d})$ acting on $\mb A^2(\mb D^d).$ Now  Theorem \ref{redsub} and Corollary \ref{selfR} immediately generalize Theorem \ref{D} to the case of the polydisc $\mb D^d.$ 
\begin{thm}
Suppose that ${\mathbf B} : \mathbb D^d \to \mb D^d$ is given by ${\mathbf B}(\bl z) = \big({\rm B}_1 (z_1),\ldots, {\rm B}_d (z_d)\big),$ where each ${\rm B}_i$ is a finite  Blaschke product of order $m_i.$ Then $\Gamma_{{\mathbf B}}\big(\mb A^2(\mb D^d)\big)$ is a non-trivial minimal reducing subspace of $\mathbf M_{\mathbf B}$ on $\mb A^2(\mb D^d).$  Moreover, the restriction of $\mathbf M_\mathbf B$ to $\Gamma_{{\mathbf B}}\big(\mb A^2(\mb D^d)\big)$ is unitarily equivalent to the Bergman multishift. Moreover,
\Bea
\Gamma_{{\mathbf  B}}\big(\mb A^2(\mb D^d)\big)=\ov{\rm span} \{J_{\mathbf B}{\rm B}^{\bl\a} : \bl\a \in \mb Z_+^d \}
\text{~and~} \Bigg\{\sqrt{\prod_{j=1}^d{\frac{\a_j+1}{m_j}}}J_{\mathbf B}\mathbf B^{\bl \a} \Bigg\}_{\bl\a\in\mb Z_+^d}
 \Eea
 is an orthonormal basis of $\Gamma_{{\rm B}}\big(\mb A^2(\mb D^d)\big).$ 
 \end{thm}
   
\

\subsection{ Example (Symmetrized Polydisc)}
Let $\bl s : \C^d \to \C^d$ be the symmetrization map defined by $\bl s(\bl z) = \big(s_1(\bl z), \ldots , s_d(\bl z) \big),$ where $s_i$ is the \emph{elementary symmetric polynomial} in $d$ variables of degree $i$, that is, $s_i$
is the sum of all products of $i$ distinct variables $z_i$ so that
$$
s_i(\bl z) = \sum_{1\leq k_1< k_2 <\ldots <k_i\leq d} z_{k_1} \cdots z_{k_i}.$$
The symmetrization map $\bl s$ is a proper holomorphic map of multiplicity $d!$ (see \cite[Theorem 5.1]{MR667790}). The domain $\mb G_d:=\bl s(\mb D^d) $ is known as the symmetrized polydisc. It is pointed out in \cite[p. 771 ]{MR3906291} that $\mb G_d$ is not a Reinhardt domain. In \cite[Theorem 1]{MR2135687}, authors characterize proper holomorphic self-maps and automorphisms of the symmetrized polydisc $\mathbb G_d.$  In particular, they proved that $\mb G_d$ admits proper holomorphic self-maps which are not automorphisms of $G_d.$
The theorem states as following.
\begin{thm}[Edigarian-Zwonek]\label{sm}
 Let $f: \mathbb G_d \to \mathbb G_d$ be a holomorphic mapping. Then $f$ is proper if and only if there exists a finite Blaschke product 
 ${\rm B}$  such that  
 \Bea
 f(\bl s(\bl z)) = \bl s \big({\rm B} (z_1),\ldots, {\rm B} (z_d)\big) \text{~for~} \bl z=(z_1,\ldots, z_d) \in \mathbb D^d,
 \Eea
 where $\bl s$ is symmetrization map. In particular, $f$ is an automorphism if and only if 
 \Bea
 f(\bl s(\bl z)) = \bl s \big(\v (z_1),\ldots, \v (z_d)\big) \text{~for~} \bl z=(z_1,\ldots, z_d) \in \mathbb D^d,
 \Eea
 where $\v$ is an automorphism of $\mb D.$
\end{thm}

Before describing reducing subspaces of $\mathbf M_f,$ we exhibit examples of proper holomorphic self-maps $f$ of $\mb G_d$ with trivial Deck($f$). 
\begin{prop}
There is a proper holomorphic self-map $f $ of $\mb G_d$ with ${\rm Deck}(f)=\{{\rm identity}\}.$
\end{prop}
\begin{proof}
Choose a proper holomorphic self-map $f$ of $\mb G_d$ given by \bea\label{Gd}
f\big(\bl s(\bl z)\big)=\bl s\big(\w{\rm B}(z_1),\ldots,\w{\rm B}(z_d)\big) \text{~for~} \bl z=(z_1,\ldots,z_d)\in\mb D^d,
\eea 
where $\w{\rm B}$ is the finite Blaschke product  given in \eqref{B2}.
If  $\v\in {\rm Deck}(f)$ for some $\v\in {\rm Aut}(\mb G_d),$ we claim that $\v={\rm identity}.$ From Theorem \ref{sm}, we know that ${\rm Aut}(\mb G_d)\cong {{\rm Aut}(\mb D)}.$ For $\bl z=(z_1,\ldots, z_d)\in\mb D^d,$ we have 
\Bea
\big(f\circ\v\big)\big(\bl s(\bl z)\big)=f\Big(\bl s\big(\v(z_1),\ldots,\v(z_d)\big)\Big)=\bl s\Big(\w{\rm B}\big(\v(z_1)\big),\ldots,\w{\rm B}\big(\v(z_d)\big)\Big).
\Eea
Since $f\circ\v=f,$ we must have 
\Bea
\big(f\circ\v\big)\big(\bl s(\bl z)\big)=\bl s\Big(\w{\rm B}\big(\v(z_1)\big),\ldots,\w{\rm B}\big(\v(z_d)\big)\Big)=\bl s\big(\w{\rm B}(z_1),\ldots,\w{\rm B}(z_d)\big).
\Eea
\end{proof}
\begin{rem}\rm
From the proof of the Proposition above it is clear that if $f$ is a proper holomorphic self-map whose associated finite Blaschke product ${\rm B}$ contains $\w{\rm B}$ as a factor then Deck($f$) is trivial.
\end{rem}

 Theorem \ref{redsub} and Corollary \ref{self} enables us to describe a non-trivial minimal reducing subspace of $\mathbf M_f$ acting on $\mb A^2(\mb G_d)$ for a proper holomorphic self-map $f$ of $\mb G_d.$
\begin{thm}\label{sym}
Suppose that $f$ is a proper holomorphic self-map of $\mb G_d.$ Then $\Gamma_f\big(\mb A^2(\mb G_d)\big)$ is a non-trivial minimal reducing subspace of $\mathbf M_f$ on $\mb A^2(\mb G_d).$  Moreover, the restriction of $\mathbf M_f$ to $\Gamma_f\big(\mb A^2(\mb G_d)\big)$ is unitarily equivalent to the Bergman operator $\mathbf M$ on $\mb A^2(\mb G_d).$
 \end{thm}

It is pointed out in \cite[Corollary 3.19]{MR3906291} that the Bergman operator $\mathbf M$ on $\mb A^2(\mb G_d)$ is not unitarily equivalent to a joint weighted shift. 
 Let $ \mb{A}^{2}_{\rm anti} (\mb{D}^d)$ be the subspace of $\mb{A}^{2}  (\mb{D}^n)$ consisting of anti-symmetric functions, that is 
 \Bea
 \mb{A}^{2}_{\rm anti} (\mb{D}^d)=\{f\in\mb A^2(\mb D^d):f\circ\sigma^{-1}={\rm sgn}(\sigma)f \text{~for~} \sigma\in\mathfrak S_d\},
 \Eea
 where $\mathfrak S_d$ is the permutation group on $d$ symbols and ${\rm sgn} (\sigma)$ is $1$ or $-1$ according as $\sigma$ is an even or an odd permutation, respectively. It follows from \cite[p. 2363]{MR3043017}  that the Bergman operator $\mathbf M$ on ${\mb A^2(\mb G_n)}$ is unitarily equivalent to $\mathbf M_{\bl s}=(M_{s_1},\ldots, M_{s_d})$ on $\mb A_{\rm anti}^2(\mathbb D^d).$ Hence from Theorem \ref{sym} we conclude the following result. 

\begin{thm}
If  $f$ is a proper holomorphic self-map of $\mb G_d$  Then the restriction of  $\mathbf M_{f} $ to $ \Gamma_{f}\big(\mb A^2(\mb G_d)\big)$ is unitarily equivalent to $\mathbf M_{\bl s}$ on $\mb A_{\rm anti}^2(\mathbb D^d).$  Consequently, $\mathbf M_{\bl s}$  on $\mb A_{\rm anti}^2(\mathbb D^d)$ is irreducible.
\end{thm}

\section{Formula for Reproducing Kernel of $\mb A^2(\Omega_2)$}
Computation of the reproducing kernel of $\Gamma_f\big(\mb A^2(\Omega_2)\big)$ for a proper holomorphic map $f:\Omega_1\to\Omega_2$ plays a crucial role in computing the Bergman kernel of $\mb A^2(\Omega_2)$ in the technique developed in \cite{MR3043017} and generalized in \cite{MR3133729}.
\begin{prop}\label{rk}
The reproducing kernel of $K_f$ of $\Gamma_f\big(\mb A^2(\Omega_2)\big)$ is given by 
\Bea
K_f(\bl z,\bl w)= \frac{1}{m}J_f(\bl z)K_2\big(f(\bl z),f(\bl w)\big)\ov{J_f(\bl w)}  \mbox{~for~} \bl z, \bl w\in\Omega_1.
\Eea

\end{prop}

\begin{proof}
Since $K_2$ is  the reproducing kernel of $\mb A^2(\Omega_2),$ it follows that
\Bea
(K_2)_{f(\bl w)}:=K_2\big(\cdot, f(\bl w)\big)\in\mb A^2(\Omega_2) \text{~for~} \bl w\in\Omega_1.
\Eea
For every  $\bl w\in\Omega_1, \frac{1}{m} \ov {J_f(\bl w)}J_f\big( (K_2)_{f(\bl w)}\circ f\big)\in \Gamma_f\big(\mb A^2(\Omega_2)\big).$
For $\psi\in \mb A^2(\Omega_2),$  we note that
\Bea
&& \inner{J_f(\psi\circ f)}{\frac{1}{m} \ov{J_f(\bl w)}J_f\big( (K_2)_{f(\bl w)}\circ f\big)}\\
&=&  {J_f(\bl w)}\inner{\Gamma_f\psi}{\Gamma_f\big((K_2)_{f(\bl w)}\big)}\\
&=&  J_f(\bl w)\inner{\psi}{(K_2)_{f(\bl w)}}\\
&=&  J_f(\bl w)(\psi\circ f)(\bl w).
\Eea
Therefore, for every  $w\in\Omega_1, \frac{1}{m} \ov{J_f(\bl w)}J_f\big( (K_2)_{f(\bl w)}\circ f\big)$ has reproducing property. By the uniqueness of the reproducing kernel of a Hilbert space  with a  reproducing kernel, we conclude that $ \frac{1}{m} \ov{J_f(\bl w)}J_f\big( (K_2)_{f(\bl w)}\circ f\big),\bl w\in\Omega_1$ is the reproducing kernel function of $\Gamma_f\big(\mb A^2(\Omega_2)\big).$  If we denote the reproducing kernel of  $\Gamma_f\big(\mb A^2(\Omega_2)\big)$ by $K_f,$ then
\Bea
K_f(\bl z,\bl w)&=&\inner{(K_f)_{\bl w}}{(K_f)_{\bl z}}\\
&=&\inner{ \frac{1}{m} \ov{J_f(w)}J_f\big( (K_2)_{f(w)}\circ f\big)}{\frac{1}{m} \ov{J_f(z)}J_f\big( (K_2)_{f(z)}\circ f\big)}\\
&=& \frac{1}{m} J_f(z)\ov{J_f(w)}\inner{\Gamma_f\big((K_2)_{f(\bl w)}\big)}{\Gamma_f\big((K_2)_{f(\bl z)}\big)}\\
&=& \frac{1}{m} J_f(\bl z)\ov{J_f(\bl w)}\inner{(K_2)_{f(\bl w)}}{(K_2)_{f(\bl z)}}\\
&=& \frac{1}{m} J_f(\bl z)K_2\big(f(\bl z),f(\bl w)\big)\ov{J_f(\bl w)}  \mbox{~for~}\bl z, \bl w\in\Omega_1.
\Eea
\end{proof}

\begin{rem}\rm
 Here is an alternative way of arriving at the formula for $K_f.$ Let $\{e_{\bl \a}\}_{\a\in\m I}$ be an orthonormal basis for $\mb A^2(\Omega_2).$ Since $\Gamma_f$ is an isometry, $\{\Gamma_f e_{\bl \a}\}_{\bl\a\in\m I}$ is an orthonormal basis for $\Gamma_f\big(\mb A^2(\Omega_2)\big).$ Therefore the reproducing kernel $K_f$ of $\Gamma_f\big(\mb A^2(\Omega_2)\big)$ is given by the following formula
\Bea
K_f(\bl z, \bl w)&=& \sum_{\bl\a\in\m I}(\Gamma_fe_{\bl\a})(\bl z)\ov{(\Gamma_fe_{\bl\a})(\bl w)}\\
&=& \frac{1}{m}\sum_{\a\in\m I}J_f(\bl z)\Big(e_{\bl\a}\big(f(\bl z)\big)\Big)\ov{J_f(\bl w)\Big(e_{\bl\a}\big(f(\bl w)\big)\Big)}\\
&=&\frac{1}{m} J_f(\bl z) \Big( \sum_{\bl\a\in\m I}e_{\bl\a}\big(f(\bl z)\big)\ov{e_{\bl\a}\big(f(\bl w)\big)}\Big)\ov{J_f(\bl w)}\\
&=&\frac{1}{m} J_f(\bl z)K_2\big(f(\bl z),f(\bl w)\big)\ov{J_f(\bl w)}  \mbox{~for~} \bl z, \bl w\in\Omega_1.
\Eea
\end{rem}
 
 \subsection{Under the Action of Pseudoreflection Groups}
 Here, we restrict to proper maps which are factored by automorphisms.  We say that a proper holomorphic map $  f:\Omega_1\to \Omega_2$ is {\it factored by automorphisms} if there exists a finite subgroup $G\subseteq {\rm Aut}(\Omega_1)$ such that
\Bea
f^{-1} f(\bl z)=\bigcup_{\rho\in G}\{\rho(\bl z)\} \,\, \text{~for~} \bl z\in \Omega_1.
\Eea
It is well-known that such a group $G$ is either a group generated by pseudoreflections or conjugate to a pseudoreflection group. 
 \begin{defn}
A pseudoreflection on $\C^d$ is a linear homomorphism $\rho: \C^d \rightarrow \C^d$ such that $\rho$ has finite order in $GL(d,\mb C)$ and the rank of $I_d - \rho$ is 1, that is, $\rho$ is not the identity map and fixes a hyperplane pointwise. A group generated by pseudoreflections is called a pseudoreflection group.
\end{defn}

For a pseudoreflection $\rho,$ fix $H_{\rho} := \ker(I_d - \rho).$ By definition, the subspace $H_{\rho}$ has dimension $d-1.$ Moreover, for $\bl z \in H_{\rho},$ one has $(I_d - \rho) \bl z = 0,$ equivalently, $ \rho \bl z = \bl z,$ that is, $\rho$ fixes the hyperplane $H_{\rho}$ pointwise. We call such hyperplanes \emph{reflecting.} Suppose the distinct reflecting hyperplanes associated to the group $G$ are $H_1,\ldots,H_t.$ For each $H_i$, there exists a cyclic subgroup $K_i$ of $G$ of order $m_i$ such that every element of $K_i$ fixes $H_i$ pointwise, that is, $\rho\bl z = \bl z,$ whenever $\bl z \in H_i, \rho \in K_i$. Each $K_i$ is generated by some pseudorefelction. Suppose the defining function of each $H_i$ is denoted by the linear form $L_i$ for $i=1,\ldots,t.$  We fix the notation $f_\mu = \prod_{i=1}^t L_i^{m_i -1 }.$ Suppose $\{\theta_i \}_{i=1}^d$ is a homogeneous system of parameters associated to the finite pseudoreflection group $G.$ Then the jacobian of the map $\bl \theta := (\theta_1,\ldots,\theta_d)$ is given by $c^{-1}f_\mu$ for some scalar $c$, by \cite[Lemma, p. 616]{MR117285}.

The linear representation $\mu : G \to \C^* $ is given by $\mu (\rho) = \det^{-1}(\rho)$ for all $\rho \in G.$ Therefore, the corresponding character of the representation $\mu,$ $\chi_{\mu} : G \to \C^*$ takes $\rho \mapsto \det^{-1}(\rho).$ Note that this character has the unique property $\chi_{\mu}(\rho_1 \rho_2) = \det^{-1} (\rho_1 \rho_2) = \det^{-1} (\rho_1) \det^{-1} (\rho_2) = \chi_{\mu}(\rho_1) \chi_{\mu}(\rho_2).$ For example, the sign representation of the symmetric group on $n$ symbols is isomorphic to the representation $\mu : \mathfrak S_n \to \C^*$.

Let $\Omega_1$ be a $G$-invariant domain in $\C^d.$ Let $\m H$ be an analytic Hilbert module on $\Omega_1$ over the ring of polynomials in $n$ variables. We assume that the reproducing kernel of $\m H$ is $G$-invariant. Consider the subspace $R^G_{\mu}(\m H) = \{f \in \m H : f \circ \rho^{-1} = \chi_\mu(\rho) f, ~ {\rm for ~~ all~} \rho \in G \}$.
We call the elements of the subspace $R^G_{\mu}(\m H)$ by $\chi_\mu$-invariant. This notion of $\chi_\mu$-invariance in $\m H$ is borrowed from the invariant theory. 
\begin{lem}
Let $f \in R_\mu^G(\m H).$ Then $f_\mu$ divides $f$ and $\frac{f}{f_\mu}$ is a $G$-invariant holomorphic function on $\Omega_1.$
\end{lem}
\begin{proof}

Let $\rho_1$ be a generator of the cyclic subgroup of $G$ whose elements fix $H_1$ pointwise. Let $m_1$ be the smallest positive integer such that $\rho_1^{m_1} = Id.$ Using a linear change of coordinates in $\Omega_1,$ we consider a new coordinate system $y_1 = L_1, y_2=x_2,\ldots,y_n=x_n.$ Then we express that $\rho_1 = {\rm diag}(\omega ,1, \ldots, 1),$ where $\omega = e^{\frac{2\pi i}{m_1}}.$ Since $f \in R_\mu^G(\m H),$ then $f(\rho_1^{-1} \cdot (y_1,\ldots,y_n)) = \chi_\mu(\rho_1) f(y_1,\ldots,y_n).$ The action of $G$ on $\mb C^n$ is given by $\rho \cdot (y_1,\ldots,y_n) = \rho^{-1}(y_1,\ldots,y_n).$ So it turns out to be 
\Bea
f(\omega y_1, y_2, \ldots, y_n) &=& {\det}^{-1}(\rho_1) f(y_1,\ldots,y_n)\\
&=& \omega^{m_1-1} f(y_1,\ldots,y_n)
.
\Eea
So $f(y_1,\ldots,y_n)$ is divisible by $y_1^{m_1-1}.$ Now changing the coordinates, we get $f(x_1,\ldots,x_n)$ is divisible by $L_1^{m_1-1}.$ Repeating this argument for each $i=2,\ldots,t,$ we have that $f_\mu$ divides $f.$
Similar arguments are given in \cite[Lemma 2.2, p. 137]{MR460484}.

 Then $\rho(f_\mu)(y_1,\ldots,y_n) = \rho(cJ_{\bl \theta})(y_1,\ldots,y_n) = c J_{\bl \theta}(\omega y_1,\dots,y_n) = \omega^{-1}(cJ_{\bl \theta})(y_1,\ldots,y_n) = \chi_\mu(\rho) f_\mu(y_1,\ldots,y_n).$ So $f_\mu$ is $\mu$-invariant.

The quotient of a $\mu$-invariant function by a $\mu$-invariant function is evidently $G$-invariant. Hence, $\frac{f}{f_\mu}$ is $G$-invariant. The holomorphicity of $\frac{f}{f_\mu}$ on $\Omega_1$ is followed by repeating the arguments given in \cite[Theorem 4.2]{BDGS}.
\end{proof}
\begin{rem}
Any $f \in R_\mu^G(\m H)$ can be written as $f = f_\mu ~(\widehat{f} \circ \bl \theta)$ for $\widehat{f} \in \mathcal O(\Omega_2).$
\end{rem}


The linear operator $\mb P_\mu : \m H \to \m H$ is given by \bea\label{projmu}
\mb P_\mu f = \frac{1}{|G|}\sum_{\rho \in G} \chi_\mu(\rho^{-1}) f \circ \rho^{-1},
\eea
$f \in \m H.$ It is known that $\mb P_\mu$ is an orthogonal projection and the range is closed, see \cite{BDGS}. It is easy to see that $R^G_{\mu}(\m H) = \mb P_\mu (\m H).$  
To see it, suppose that $f \in \mb P_\mu (\m H)$ and $\rho_0 \in G.$ Note that $f \circ \rho_0^{-1} = \mb P_\mu(f) \circ \rho_0^{-1} = (\frac{1}{|G|} \sum_{\rho \in G} \chi_\mu(\rho^{-1}) f \circ \rho^{-1}) \circ \rho_0^{-1} = \frac{1}{|G|} \sum_{\rho \in G} \chi_\mu(\rho^{-1}) f \circ \rho^{-1} \circ \rho_0^{-1} = \frac{1}{|G|} \sum_{\eta \in G} \chi_\mu(\eta^{-1}\rho_0 ) f \circ \eta^{-1} = \chi_\mu(\rho_0) f.$ 
Conversely, for $f \in R^G_{\mu}(\m H),$ we get
$\mb P_\mu (f) = \frac{1}{|G|} \sum_{\rho \in G}  \chi_\mu(\rho^{-1}) f \circ \rho^{-1} = \frac{1}{|G|} \sum_{\rho \in G}  \chi_\mu(\rho^{-1}) \chi_\mu(\rho) f =\frac{1}{|G|} \sum_{\rho \in G}  \chi_\mu(\rho^{-1} \rho) f.$ Clearly, for any $\rho \in G,$ $\chi_\mu(\rho^{-1} \rho) = \chi_\mu(Id)= \det^{-1}(Id) = 1.$ So $\mb P_\mu (f) = \frac{1}{|G|} \sum_{\rho \in G} f = f.$ This proves the claim. 

Since $\Omega_1$ is $G$-invariant and $G$ is a subgroup of the unitary operators on $\C^d$, the Bergman kernel on $\Omega_1$ is also $G$-invariant. Therefore, we have $R^G_{\mu}(\mb A^2(\Omega_1)) = \mb P_\mu (\mb A^2(\Omega_1)).$

Recall that $\{\theta_i\}_{i=1}^d$ is a homogeneous system of parameters (h.s.o.p) associated to the pseudoreflection group $G$. We define associated polynomial map by ${\bl\theta}: \C^d \rightarrow \C^d$, where
\begin{equation*}
{\bl\theta}(\bl z) = \big(\theta_1(\bl z),\ldots,\theta_d(\bl z)\big),\,\,\bl z\in\C^d.
\end{equation*} 
\begin{prop}
Let $\Omega_1$ be a $G$-invariant domain. Then 
\begin{enumerate}
    \item[(i)] $\bl \theta(\Omega_1)$ is a domain, and
    \item[(ii)] $\bl \theta : \Omega_1 \to \Omega_2$ is a proper map, where $\Omega_2 := \bl \theta(\Omega)$.
\end{enumerate}
\end{prop}
A proof can  be found in \cite[Proposition 1, p. 556]{MR3133729}. We recall the linear map defined in \eqref{genp} for this particular case. The map changes to $\Gamma_{\bl \theta} : \mb A^2\big(\Omega_2\big) \to \mb A^2(\Omega_1)$ by $\Gamma_{\bl \theta} \phi = \frac{1}{\sqrt{m}} J_{\bl \theta} (\phi \circ \bl \theta)$, where $J_{\bl \theta}$ is the complex jacobian of the map $\bl \theta$ and $m$ is the multiplicity of $\bl \theta.$ In this case, the multiplicity $m$ is the order of the group $G.$ 
Our goal is to identify the range $\Gamma_{\bl \theta} \big(\mb A^2\big(\Omega_2\big)\big)$ with the subspace $\mb P_{\mu}\big(\mb A^2(\Omega_1)\big)$. This allows us to view the reproducing kernel of $\Gamma_{\bl \theta} \big(\mb A^2\big(\Omega_2\big)\big)$ in a more convenient way.
\begin{rem} \rm
Note that we already have an orthogonal projection $P$ for the subspace $\Gamma_{\bl \theta} \big(\mb A^2\big(\Omega_2\big)\big).$ Initially, the expression of $P$ in the Equation \eqref{projec} seems different to that of $P_\mu$ in the Equation \eqref{projmu}. It is known that the group of Deck transformations of the polynomial map $\bl \theta$ is $G.$ Suppose ${\bl \theta}_k$'s are local inverses of $\bl \theta$, for $k=1,\ldots,m.$ One can check that each ${\bl \theta}_k$ can be holomorphically extended to $\Omega_2.$ We denote the functions ${\bl \theta}_k \circ \bl \theta$ by $\tilde{\bl \theta}_k$.  Note that $\bl \theta \circ \tilde{\bl \theta}_k = \bl \theta$, for all $k=1,\ldots,m.$ Therefore, each $\tilde{\bl \theta}_k$ is in $G.$ Moreover, $J_{\tilde{\bl \theta}_k} = \chi_\mu(\tilde{\bl \theta}^{-1}_k)$ for all $k$, which implies that $P$ and $P_\mu$ are essentially same in this instance. We find this discussion rather ambiguous. So we approach to prove it in a more comprehensible way.
\end{rem}
\begin{thm}\label{equiv}
Let $G$ be a finite pseudoreflecion group and $\Omega_1$ be a $G$-invariant domain in $\C^d.$ Suppose $\{\theta_i\}_{i=1}^d$ is a homogeneous system of parameters associated to the group $G$ and the corresponding polynomial map is given by $\bl \theta.$ Then $\Gamma_{\bl \theta} \big(\mb A^2\big(\Omega_2\big)\big) = \mb P_{\mu}\big(\mb A^2(\Omega_1)\big).$ 
\end{thm}



\begin{proof}
Let $h \in \mb P_\mu (\mb A^2(\Omega_1)).$ Then $h = f_\mu ~\widehat{h} \circ \bl \theta$ for $\widehat{h} \in \mathcal O(\Omega_1).$ Clearly, $$\norm{h}^2_{\Omega_1} = \int_{\Omega_1} |h|^2 dA = \int_{\Omega_1} |f_\mu|^2 |\widehat{h} \circ \bl \theta|^2 dA = m \int_{\Omega_2} |\widehat{h}|^2 dA.$$ Since $\norm{h}^2_{\Omega_1} < \infty,$ evidently $\widehat{h} \in \mb A^2(\Omega_2).$ The image of $\sqrt{m}~ \widehat{h}$ under $\Gamma_{\bl \theta}$ is $h.$ So $h \in \Gamma_{\bl \theta} \big(\mb A^2(\Omega_2)\big).$

On the other hand, $\rho(\Gamma_{\bl \theta}h)(\bl z) =  \frac{1}{\sqrt{m}} J_{\bl \theta}(\rho^{-1} \cdot \bl z) ~ (f \circ \bl \theta)(\rho^{-1} \cdot \bl z) = \chi(\rho) J_{\bl \theta}(\bl z) ~ (f \circ \bl \theta)(\bl z) = \chi(\rho) (\Gamma_{\bl \theta} h)(\bl z) .$ That is $\Gamma_{\bl \theta} h \in \mb P_\mu (\mb A^2(\Omega_1))$ for all $h \in \mb A^2(\Omega_2).$
\end{proof}

The reproducing kernel of $\mb A^2(\Omega_i)$ is denoted by $K_i,$ for $i=1,2.$ In the following discussion, we determine an expression of $K_2$ in terms of $K_1.$ The Proposition \ref{rk} describes a formula for $K_2$. It states that
the reproducing kernel of $K_\mu$ of $\Gamma_{\bl \theta}\big(\mb A^2(\Omega_2)\big)$ is given by 
\bea\label{first}
K_\mu(\bl z,\bl w)= \frac{1}{m}J_{\bl \theta}(\bl z)K_2\big(\bl \theta (\bl z),\bl \theta (\bl w)\big)\ov{J_{\bl \theta}(\bl w)}  \mbox{~for~} \bl z, \bl w\in\Omega_1.
\eea

Since $\mb P_\mu$ is the orthogonal projection corresponding to the subspace $\Gamma_{\bl \theta}\big(\mb A^2(\Omega_2)\big),$ the reproducing kernel $K_\mu(\bl z, \bl w) = \inner{\mb P_\mu (K_1)_{\bl w}}{(K_1)_{\bl z}}.$ Then 
\bea\label{sec}
K_\mu(\bl z, \bl w) = \big( \mb P_\mu (K_1)_{\bl w} \big) (\bl z) &=&\nonumber \frac{1}{|G|} \sum_{\sigma \in G}\chi_\mu(\sigma^{-1})  (K_1)_{\bl w}\circ \sigma^{-1}(\bl z) \\&=&\nonumber \frac{1}{|G|} \sum_{\sigma \in G}\chi_\mu(\sigma^{-1})  (K_1)_{\bl w}(\sigma^{-1} \cdot \bl z) \\ &=& \frac{1}{|G|} \sum_{\sigma \in G}\chi_\mu(\sigma^{-1})  K_1(\sigma^{-1} \cdot \bl z, \bl w ).
\eea

Combining the expression of \eqref{first} and \eqref{sec}, we state the next proposition. \begin{prop}\label{kernel}
The reproducing kernel of $\mb A^2(\Omega_i)$ is denoted by $K_i,$ for $i=1,2.$
 Then for $\bl z, \bl w \in \Omega_1,$ \bea\label{thir}K_2(\bl \theta(\bl z), \bl \theta(\bl w)) = J_{\bl \theta}^{-1}(\bl z)\big( \displaystyle\sum_{\sigma \in G} \chi_\mu(\sigma^{-1} ) K_1(\sigma^{-1} \cdot \bl z, \bl w)  \big) \overline{J_{\bl \theta}^{-1}(\bl w)}.\eea
\end{prop} 
Now, we settle the ambiguity of the expression \eqref{thir} for $\bl z, \bl w$ in $N$, where $N \subseteq \Omega_1$ is the zero set of the function $J_{\bl \theta}$.
Note that $K_\mu(\bl z, \bl w) = \big( \mb P_\mu (K_1)_{\bl w} \big) (\bl z) = J_{\bl \theta}(\bl z) \big((K_1)_{\bl w}^\mu \circ \bl \theta(\bl z)\big)$ for some $(K_1)_{\bl w}^\mu \in \mathcal O(\Omega_2)$ for a fixed but arbitrary $\bl w \in \Omega_1$. Similarly, $K_\mu(\bl z, \bl w) = \big(\overline{ \mb P_\mu (K_1)_{\bl z}} \big) (\bl w) = \overline{J_{\bl \theta}(\bl w)} \big( \overline{(K_1)_{\bl z}^\mu \circ \bl \theta}(\bl w)\big)$ for some $(K_1)_{\bl z}^\mu \in \mathcal O(\Omega_2),$ when $\bl z$ is fixed but arbitrary. The variables $\bl z$ and $\bl w$ are independent of each other in $K_\mu(\bl z, \bl w)$, so it is divisible by $J_{\bl \theta}(\bl z) \overline{J_{\bl \theta}(\bl w)}$ for every $\bl z, \bl w \in \Omega_1.$ Therefore $J_{\bl \theta}^{-1}(\bl z) K_\mu(\bl z, \bl w) \overline{J_{\bl \theta}^{-1}(\bl w)}$ is well-defined, even if $\bl z$ or $\bl w$ belongs to $N.$

\begin{ex}
Suppose $G=\mathfrak S_n$. The elementary symmetric polynomials of degree $k$ in $n$ variables $s_k$ form a homogeneous system of parameters corresponding to the group $\mathfrak S_n,$ where $k=1,\ldots,n.$ Hence, we take the symmetrization map $\bl s :=(s_1,\ldots,s_n) : \mb D^n \to \mb G_n$ in the place of $\bl \theta$ in Proposition \ref{kernel}.  The Bergman kernel $K_1$ is $\mathfrak S_n$-invariant, that is, $K_1(\sigma \cdot \bl z, \sigma \cdot \bl w) = K_1(\bl z, \bl w)$ for $\sigma \in \mathfrak S_n$ and $\bl z, \bl w \in \mb D^n.$ This implies 
\Bea
K_2(\bl s(\bl z), \bl s(\bl w)) &=& J_{\bl s}^{-1}(\bl z)\big( \displaystyle\sum_{\sigma \in \mathfrak S_n} \chi_\mu(\sigma^{-1} ) K_1(\sigma^{-1} \bl z, \bl w)  \big) \overline{J_{\bl s}^{-1}(\bl w)}\\ &=& J_{\bl s}^{-1}(\bl z)\big( \displaystyle\sum_{\sigma \in \mathfrak S_n} {\rm sgn}(\sigma^{-1} ) K_1( \bl z, \sigma \cdot \bl w)  \big) \overline{J_{\bl s}^{-1}(\bl w)} \\  &=& J_{\bl s}^{-1}(\bl z)\big( \displaystyle\sum_{\sigma \in \mathfrak S_n} {\rm sgn}(\sigma^{-1} ) \prod_{i=1}^n(1-z_i\bar{w}_{\sigma^{-1}(i)})^{-2}  \big) \overline{J_{\bl s}^{-1}(\bl w)} \\ &=& J_{\bl s}^{-1}(\bl z) \Big( \det \big(\big( (1 - z_i \bar{w_j})^{-2}\big)\big)_{i,j =1 }^{n} \Big) \overline{J_{\bl s}^{-1}(\bl w)}. 
\Eea
This expression of Bergman kernel of symmetrized polydisc in terms of Bergman kernel of polydisc has been derived in \cite[Proposition 9., p. 369]{MR2135687}.
\end{ex}

\section*{acknowledgement}
The results of this article form a part of my PhD Thesis at Indian Institute of Science Education and Research Kolkata. I would like to thank my thesis supervisor S. Shyam Roy for his numerous helpful comments and suggestions regarding the material of this article.

\bibliographystyle{siam}
\bibliography{Bibliography.bib}

\begin{thebibliography}{10}

\bibitem{MR733691}
{\sc E.~Bedford}, {\em Proper holomorphic mappings}, Bull. Amer. Math. Soc.
  (N.S.), 10 (1984), pp.~157--175.

\bibitem{MR807258}
{\sc E.~Bedford and J.~Dadok}, {\em Proper holomorphic mappings and real
  reflection groups}, J. Reine Angew. Math., 361 (1985), pp.~162--173.

\bibitem{BDGS}
{\sc S.~Biswas, S.~Datta, G.~Ghosh, and S.~Shyam~Roy}, {\em A
  {C}hevalley-{S}hephard-{T}odd theorem for analytic {H}ilbert module},
  https://arxiv.org/abs/1811.06205,  (2018).

\bibitem{MR3906291}
{\sc S.~Biswas, G.~Ghosh, G.~Misra, and S.~Shyam~Roy}, {\em On reducing
  submodules of {H}ilbert modules with {$\mathfrak{S}_n$}-invariant kernels},
  J. Funct. Anal., 276 (2019), pp.~751--784.

\bibitem{MR3188714}
{\sc S.~Biswas and S.~Shyam~Roy}, {\em Functional models of
  {$\Gamma_n$}-contractions and characterization of {$\Gamma_n$}-isometries},
  J. Funct. Anal., 266 (2014), pp.~6224--6255.

\bibitem{MR2142182}
{\sc C.~Costara}, {\em On the spectral {N}evanlinna-{P}ick problem}, Studia
  Math., 170 (2005), pp.~23--55.

\bibitem{MR1315866}
{\sc R.~L. Craighead, Jr. and F.~W. Carroll}, {\em A decomposition of finite
  {B}laschke products}, Complex Variables Theory Appl., 26 (1995),
  pp.~333--341.

\bibitem{MR1131852}
{\sc G.~Dini and A.~Selvaggi~Primicerio}, {\em Proper holomorphic mappings
  between generalized pseudoellipsoids}, Ann. Mat. Pura Appl. (4), 158 (1991),
  pp.~219--229.

\bibitem{MR2135687}
{\sc A.~Edigarian and W.~Zwonek}, {\em Geometry of the symmetrized polydisc},
  Arch. Math. (Basel), 84 (2005), pp.~364--374.

\bibitem{MR3363367}
{\sc K.~Guo and H.~Huang}, {\em Multiplication operators on the {B}ergman
  space}, vol.~2145 of Lecture Notes in Mathematics, Springer, Heidelberg,
  2015.

\bibitem{MR2503238}
{\sc K.~Guo, S.~Sun, D.~Zheng, and C.~Zhong}, {\em Multiplication operators on
  the {B}ergman space via the {H}ardy space of the bidisk}, J. Reine Angew.
  Math., 628 (2009), pp.~129--168.

\bibitem{MR2068435}
{\sc J.~Hu, S.~Sun, X.~Xu, and D.~Yu}, {\em Reducing subspace of analytic
  {T}oeplitz operators on the {B}ergman space}, Integral Equations Operator
  Theory, 49 (2004), pp.~387--395.

\bibitem{MR3902057}
{\sc H.~Huang and P.~Ling}, {\em Joint reducing subspaces of multiplication
  operators and weight of multi-variable {B}ergman spaces}, Chin. Ann. Math.
  Ser. B, 40 (2019), pp.~187--198.

\bibitem{HZ}
{\sc H.~Huang and D.~Zheng}, {\em Multiplication operators on the {B}ergman
  space of bounded domains in {$\mathbb C^d$}},
  https://arxiv.org/abs/1511.01678v1,  (2015).

\bibitem{L}
{\sc J.~Lebl}, {\em Tasty bits of several complex variables},
  https://www.jirka.org/scv/,  (2020).

\bibitem{MR3043017}
{\sc G.~Misra, S.~Shyam~Roy, and G.~Zhang}, {\em Reproducing kernel for a class
  of weighted {B}ergman spaces on the symmetrized polydisc}, Proc. Amer. Math.
  Soc., 141 (2013), pp.~2361--2370.

\bibitem{MR1324108}
{\sc R.~Narasimhan}, {\em Several complex variables}, Chicago Lectures in
  Mathematics, University of Chicago Press, Chicago, IL, 1995.
\newblock Reprint of the 1971 original.

\bibitem{MR1803086}
{\sc R.~Narasimhan and Y.~Nievergelt}, {\em Complex analysis in one variable},
  Birkh\"{a}user Boston, Inc., Boston, MA, second~ed., 2001.

\bibitem{MR667790}
{\sc W.~Rudin}, {\em Proper holomorphic maps and finite reflection groups},
  Indiana Univ. Math. J., 31 (1982), pp.~701--720.

\bibitem{MR2446682}
\leavevmode\vrule height 2pt depth -1.6pt width 23pt, {\em Function theory in
  the unit ball of {$\Bbb C^n$}}, Classics in Mathematics, Springer-Verlag,
  Berlin, 2008.
\newblock Reprint of the 1980 edition.

\bibitem{MR460484}
{\sc R.~P. Stanley}, {\em Relative invariants of finite groups generated by
  pseudoreflections}, J. Algebra, 49 (1977), pp.~134--148.

\bibitem{MR117285}
{\sc R.~Steinberg}, {\em Invariants of finite reflection groups}, Canadian J.
  Math., 12 (1960), pp.~616--618.

\bibitem{MR1969797}
{\sc M.~Stessin and K.~Zhu}, {\em Generalized factorization in {H}ardy spaces
  and the commutant of {T}oeplitz operators}, Canad. J. Math., 55 (2003),
  pp.~379--400.

\bibitem{MR3133729}
{\sc M.~Trybula}, {\em Proper holomorphic mappings, {B}ell's formula, and the
  {L}u {Q}i-{K}eng problem on the tetrablock}, Arch. Math. (Basel), 101 (2013),
  pp.~549--558.

\end{thebibliography}

\end{document}